\documentclass[12pt, reqno]{amsart}
\usepackage{amsmath}
\usepackage{amsthm}
\usepackage{amsfonts}
\usepackage{amssymb}
\usepackage{verbatim}
\usepackage{graphicx}
\usepackage[margin=1.0in]{geometry}
\usepackage{url}
\usepackage{enumerate}
\usepackage[pdftex,bookmarks=true]{hyperref}
\usepackage[normalem]{ulem}
\usepackage[usenames, dvipsnames]{color}

\theoremstyle{plain}
\newtheorem{theorem}{Theorem}[section]
\newtheorem{conjecture}[theorem]{Conjecture}

\newtheorem{lemma}[theorem]{Lemma}

\theoremstyle{definition}
\newtheorem{definition}[theorem]{Definition}

\newtheorem{remark}[theorem]{Remark}

\newcommand\pp{{\mathbb P}}
\renewcommand\P{{\mathbb P}}

\newcommand\cN{{\mathcal N}}

\newcommand\N{{\mathbb N}}

\newcommand\R{{\mathbb R}}

\newcommand\E{{\mathbb E}}
\newcommand\I{{\textbf{1}}}

\newcommand{\cov}{\textrm{Cov}}

\newcommand{\sech}{\textrm{sech}}

\newcommand{\eps}{\varepsilon}

\newcommand\sign{{\text{sign}}}

\newcommand\norm[1]{\|#1\|}

\newcommand{\ep}{\varepsilon}

\def\be#1{\begin{equation*}#1\end{equation*}}
\def\ben#1{\begin{equation}#1\end{equation}}
\def\bea#1{\begin{eqnarray*}#1\end{eqnarray*}}
\def\bean#1{\begin{eqnarray}#1\end{eqnarray}}

\def\ba#1{\begin{align*}#1\end{align*}}

\DeclareMathOperator{\Var}{Var}
\DeclareMathOperator{\Cov}{Cov}

\def\be#1{\begin{equation*}#1\end{equation*}}
\def\ben#1{\begin{equation}#1\end{equation}}
\def\bea#1{\begin{eqnarray*}#1\end{eqnarray*}}
\def\bean#1{\begin{eqnarray}#1\end{eqnarray}}

\title{Concentration inequalities for the number of real zeros of Kac polynomials}

\author{ Van Hao Can \and Oanh Nguyen}
\address{Institute of Mathematics\\ Vietnam Academy of Science and Technology\\ 18 Hoang Quoc Viet, Cau Giay, Hanoi, Vietnam }
\email{cvhao@math.ac.vn}
\thanks{The work of Van Hao Can is supported by the Vietnam Academy of Science and Technology grant number CTTH00.02/22-23. }
\address{Division of Applied Mathematics\\ Brown University\\  Providence, RI 02906, USA}
\email{oanh\_nguyen1@brown.edu}
\thanks{Oanh Nguyen is supported by NSF grants DMS–1954174 and DMS–2246575.}

\begin{document}
	
	\maketitle
	
	  \begin{abstract}  We study concentration inequalities for the number of real roots of the classical Kac polynomials $$f_{n}  (x) = \sum_{i=0}^n  \xi_i x^i$$ where $\xi_i$ are independent random variables with mean 0, variance 1, and uniformly bounded $(2+\ep_0)$-moments. We establish polynomial tail bounds, which are optimal, for the bulk of roots. For the whole real line, we establish sub-optimal tail bounds.
	 \end{abstract}

	 \section{Introduction}
	 The Kac polynomials defined as
	 $$f_{n}  (x) = \sum_{i=0}^n  \xi_i x^i$$ where $\xi_i$ are independent (but not necessarily identically distributed) random variables with mean 0 and variance 1
	 are one of the most studied classes of random polynomials. We assume these conditions on the $\xi_i$ throughout the paper.

	 The investigation of real roots of these random polynomials has a long history dated back to the beautiful work of Bloch and P\'olya \cite{BP} in 1932 where they showed that the number of real roots, denoted by $N_n(\R)$, is bounded by $\sqrt n$ with high probability, in the case that the $\xi_i$ take values in $\{\pm 1, 0\}$ with equal probability.
	 In the early 1940s, Littlewood and Offord \cite{LO1, LO2, LO3} showed that the number of real roots is in fact poly-logarithmic in $n$.
	 The exact asymptotics of the expected number of real roots was discovered by Kac \cite{Kac1943average} using his celebrated formula, nowadays known as the Kac-Rice formula. He showed that when the random variables are Gaussian, the expected number of real roots is 
	 $\left (\frac{2}{\pi} +o(1)\right ) \log n.$

	 The general case when the $\xi_i$ are not necessarily Gaussian was settled by Erd\H{o}s and Offord \cite{EO}  and Ibragimov and Maslova \cite{Ibragimov1971expected1, Ibragimov1971expected2} who showed that the same asymptotics holds when the $\xi_m$ are iid and belong to the domain of attraction of the normal law. They employed a beautiful method that related the number of real roots to the number of sign changes of $f_n$ which effectively reduced the task to evaluating certain integral of the characteristic function of the $\xi_i$. In recent years, using a new universality method, Tao and Vu \cite{TVpoly} and subsequent papers by the second authors and collaborators \cite{DOV}, \cite{nguyenvurandomfunction17} have built up a robust framework that among other things, refines the above estimate to $\frac{2}{\pi}\log n+ O(1)$ when the random variables $\xi_i$ are not necessarily identically distributed but have uniformly bounded $(2+\ep_0)$-moments for some $\ep_0>0$.
	 
	 Besides, the asymptotics distribution of $N_n(\R)$ was established in \cite{Mas2} and \cite{nguyenvuCLT} where it was shown that 
	 $$\frac{N_n(\R) - \E N_n(\R)}{c\sqrt {\log n}}$$
	 converges to the standard Gaussian in distribution as $n\to \infty$ for some constant $c$.

	 The quest for quantitative estimates of the concentration of $N_n(\R)$ has attracted a lot of attention. In the series of seminal papers \cite{LO1, LO2, LO3} back in the 1940s, Littlewood and Offord showed that if the $\xi_i$ are iid uniform in $[-1, 1]$, or standard Gaussian, or Rademacher (uniform in $\{-1, 1\}$), then 
	 $$\P\left (N_n(\R)>25(\log n)^{2}\right )\le \frac{12\log n}{n}\quad \text{and}\quad \P\left (N_n(\R)<\frac{\alpha \log n}{(\log\log n)^{2}}\right )<\frac{A}{\log n}.$$
	 
	  It is natural to conjecture the following.
	 \begin{conjecture}\label{conj} Under mild conditions, for every $\ep\in (0, 2/\pi)$, there exist constants $A, c, C$ such that for all $n$, it holds that 
	 \begin{equation} 
	 \label{eq:conj}
	 \frac1A	n^{-C} \leq \pp \left( |N_{n}(\R) - \E N_n(\R)| \geq  \ep  \log n \right) \leq A n^{-c}.
	 \end{equation}
	 \end{conjecture}
	 
	 Several interesting results (and conjectures) have been made in recent years toward this question. The left-hand side of \eqref{eq:conj} was established in the seminal paper \cite{Dembopersistence} by Dembo, Poonen, Shao and Zeitouni (2002) where, among other things, it was shown that the probability of the polynomial having no real roots is $n^{-\Theta(1)}$. 
	 \begin{theorem}  \label{thm:D} \cite[Theorem 1.2]{Dembopersistence}
	 	Assume that the $\xi_i$ are iid with all moments finite. Then there exists a positive constant  $b$ such that for all $\ep\in (0, 2/\pi)$ and for all $n$,
	 	\be{
	 		\pp\left( N_n(\R) \le \E N_n(\R) -\ep  \log n \right) \ge \P\left (N_n(\R)=0\right ) = n^{-b+o(1)}.
	 	} 
	 \end{theorem}	
	 Note that the same holds true if we replace $N_n(\R)$ by $N_n([0, 1])$, $N_n([1, \infty))$, $N_n([-1, 0])$, and $N_n((-\infty, -1])$ where $N_n(S)$ denotes the number of roots of $f_n$ in the set $S$ and the constant $b$ replaced by $b/4$. 
	  The exponent $b$, known as the {\it persistence exponent}, was characterized in terms of the centered stationary Gaussian process $Y_t$ with correlation function $R_y(t) = \sech(t/2)$ by
	  $$b = -4\lim_{T\to\infty} \frac1T \log \P\left (\sup_{0\le t\le T} Y_t\le 0\right ).$$
 	  The exact value $b=\frac34$ was established in \cite[Equation 6]{poplavskyi2018exact} by Poplavskyi and Schehr (2018) using the connection between
the Kac polynomials and the truncated real
orthogonal ensemble of random matrices. Using a different method which related the process $Y_t$ to a translationally invariant
Pfaffian point process, FitzGerald, Tribe, and Zaboronski (2022) established the same value \cite[Proposition 5]{fitzgerald2022asymptotic}.  
	 In proving these results, the authors of \cite{Dembo} (and \cite{fitzgerald2022asymptotic, poplavskyi2018exact}) used the strong approximation results of Koml\'os-Major-Tusn\'ady which necessitated the conditions that the random variables are identically distributed and all  moments are finite. More general results for the class of generalized Kac polynomials were established in \cite{Dembo} by Dembo and Mukherjee (2015).

	  In \cite[Equation 5]{schehr2008real}, Schehr and Majumdar (2008) put forward the following conjecture which is strongly related to Conjecture \ref{conj}:
	  $$\P(N_n([0, 1]) = k) \propto n^{-\tilde \phi(k/\log n) },$$
	  where $k\approx c\log n$ for some fixed $c$. The large deviation function $\tilde \phi(x)$ was computed explicitly in \cite[Equation 51]{poplavskyi2018exact} (which can be found in the arxiv version {\url{https://arxiv.org/pdf/1806.11275.pdf}}). Its asymptotic behaviors were given in \cite[Equation 6]{poplavskyi2018exact}:
	  $$\tilde \phi(c) \sim \begin{cases}
	  \frac{3}{16} + c\log c, \quad c\to 0,\\
	  \frac{1}{2\sigma^{2}}\left (c - \frac{1}{2\pi}\right )^{2}, \quad \left |c - \frac{1}{2\pi}\right |\ll 1, \sigma^{2} = \frac{1}{\pi} - \frac{2}{\pi^2},\\
	  \frac{\pi^2}{4}c^2 - \frac{\log 2}{2}c, \quad c\to \infty.
	  \end{cases}$$

	 \vspace{2mm}
	 In this paper, we tackle the right-hand side of Conjecture \ref{conj} under a less stringent condition than \cite{Dembo}. For a constant $\ep_0>0$, we say that the random variables $\xi_i$ have uniformly bounded $(2+\ep_0)$-moments if there exists a constant $C_0$ such that for all $i$, 
	 $$\E |\xi_i|^{2+\ep_0}\le C_0.$$ 
	 The main result of our paper is to establish the right-hand side of Conjecture \ref{conj} for the bulks of roots. Consider sub-intervals of $[0, 1]$ of the form
	 \be{
	 	I_{n^a} = [1-n^{-a}, 1]
	 }
	 where $a$ is a positive constant. It is well-known (e.g. \cite{DOV, Ibragimov1968average}) that while the expected number of roots in $[0,1]$ is $\left (\frac{1}{2\pi}+o(1)\right )\log n$,  the complement $[0, 1-n^{-a})$ of $I_{n^a}$ only contains $O(a\log n)$ real roots, on average. So, if $a$ is a small number, $I_{n^a}$ indeed contains most of the roots in $[0, 1]$ (see Figure \ref{fig:1})). 
	 For these interval, we prove polynomial tail bounds.
	 \begin{theorem} \label{thm:l_a} Assume that the $\xi_i$ have uniformly bounded $(2+\ep_0)$-moments for some $\ep_0>0$. Let $a$ be any positive constant. Then for all $\ep>0$, there exist constants $c, C$ such that for all $n$,
	 	\[  \P \left( \left | N_{f_n}(I_{n^{a}})-  \E N_{f_n}(I_{n^a}) \right |\ge \ep \log n \right) \le C n^{-c}.\]
	 	The same result can be established for the other intervals $[-1, 0]$, $[1, \infty)$ and $(-\infty, -1]$ where the corresponding bulk intervals are $-I_{n^a}=\{-x: x\in I_{n^a}\}$, $I_{n^a}^{-1}:=\{x^{-1}: x\in I_{n^a}\}$ and $-I_{n^a}^{-1}$.
	 \end{theorem}
	 We in fact prove a more general result in Theorem \ref{thm:l_ml} for a larger class of bulk intervals where we allow $n^{a}$ to be replaced by any parameter $m$, depending on $n$, that satisfies $(\log n)^{A}\ll m\ll n$. 
		 \begin{figure}[h!]\label{fig:1}
		\includegraphics[width=.8\textwidth]{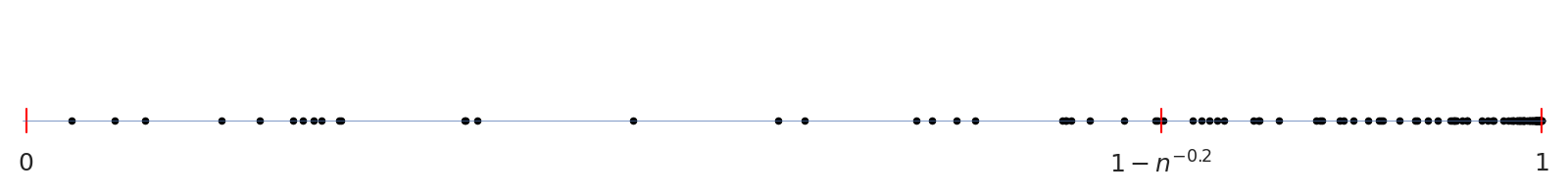}
		\caption{Real roots in $[0, 1]$ of 100 sampled Kac polynomials of degree 1000 and standard Gaussian coefficients. The real roots concentrate in the bulk $[1-n^{-0.2}, 1]$.}
	\end{figure}
	
	 \vspace{5mm}
	
	Moving to the whole real line, Theorem \ref{thm:l_a} (and Theorem \ref{thm:l_ml}) allows us to show that $n^{-\Theta(1)}$ is in fact the right order of the lower tail. 	 
	\begin{theorem}  \label{thm:lower}
	 	Assume that the $\xi_i$ have uniformly bounded $(2+\ep_0)$-moments for some $\ep_0>0$. Then for any $\ep>0$, there exist positive constants $c$ and $C$ such that for all $n$,
	 	\be{
	 		\pp\left( N_n(\R) \le \E N_n(\R) -\ep  \log n \right)  \le Cn^{-c}.
	 	} 
	 	The same holds if we replace $N_n(\R)$ by $N_n([0, 1])$, $N_n([1, \infty))$, $N_n([-1, 0])$, and $N_n((-\infty, -1])$.
	 \end{theorem}	
	 
	 For the upper tail, we show the following sub-optimal upper bound.
	 \begin{theorem} \label{thm:upper} [Upper tail]
	 	Assume that the $\xi_i$ have uniformly bounded $(2+\ep_0)$-moments for some $\ep_0>0$. Then for any $\ep>0$, there exist positive constants $c$ and $C$ such that for all $n$,
	 	\be{
	 		\pp\left( N_n(\R)  \geq \E N_n(\R) +\ep  \log n \right) \leq C\exp \left(-c \sqrt{\log n}\right).
	 	} 
	 	The same holds if we replace $N_n(\R)$ by $N_n([0, 1])$, $N_n([1, \infty))$, $N_n([-1, 0])$, and $N_n((-\infty, -1])$.
	 \end{theorem}	
 In the next section, we will outline several key difficulties in targeting the conjectured polynomial upper-tail bounds. We believe that new tools and ideas are required to resolve these problems.

	\subsection{Related literature}\label{sec:lit}

Aside from the Kac polynoials, in \cite{Basuconcentration}, Basu, Dembo, Feldheim and Zeitouni showed concentration of the number of zeros of stationary Gaussian processes in a given interval $[0, T]$.

In \cite{nguyen2019exponential}, Hoi Nguyen and Zeitouni studied random trigonometric polynomials 
$$P_n(x) = \sum_{k=1}^{n} \xi_k \cos(kx)+\xi'_k \sin(kx)$$
and showed that 
$$\P(|N_n[0, 2\pi] -\E N_n[0, 2\pi]|\ge \ep\E N_n)\le \exp(-cn)$$
where $N_n[0, 2\pi]$ is the number of real roots in $[0, 2\pi]$ and it is known that $\E N_n =\Theta(n)$. In \cite{nguyen2023concentration}, Hoi Nguyen showed concentration for the number of intersections between random eigenfunctions of general eigenvalues and a given smooth curve on flat tori.

\subsection{Difficulties and main techniques} Consider the bulk intervals $I = [1-m^{-1}, 1]$ and the ``edge" intervals $J = [0, 1 - m^{-1})$. In the following, we outline several key difficulties and our approaches for each intervals.
\begin{itemize}
	\item The bulk intervals $I = [1-m^{-1}, 1]$. In the aforementioned works \cite{Basuconcentration}, \cite{nguyen2023concentration} and \cite{nguyen2019exponential}, the number of real roots are, roughly speaking, equally spaced out. For instance, for a given interval $[0, T]$, the expected number of real roots of a stationary Gaussian process is $\Theta(T)$; and in fact, if we divide $[0, T]$ into $\Theta(T)$ equal sub-intervals, each of which contains $\Theta(1)$ roots. The same holds for random trigonometric polynomials and random nodal intersections. The roots of the Kac polynomials, on the other hand, are not equally spaced. Indeed, the interval $[0, 1- \frac1C]$ only contains $\Theta(1)$ real roots, while for each $\delta\gg 1/n$, the dyadic interval $[1-\delta, 1-\delta/2]$ also has $\Theta(1)$ roots.
	
	 Nevertheless, for the bulk intervals, we demonstrate that when appropriately rescaled, the roots are not significantly different from those of a stationary Gaussian process $(Z_t)$. This allows us to utilize the results in \cite{Basuconcentration} (see Lemma \ref{lm:donz}). More specifically, we divide the bulk intervals into tiny intervals and approximate the number of real roots in each of those intervals by the number of sign changes at the end points. Thanks to \cite{Basuconcentration}, we obtain the concentration of the latter for $(Z_t)$. The challenge now lies in demonstrating that the distribution of the number of sign changes for $f_n$ closely resembles that of $(Z_t)$ (see Lemma \ref{lm:ngzs}) which turned out to be quite challenging to get the right order of the error terms. 
	 
	 To this end, we first use universality results in \cite{nguyenvuCLT} to narrow the question down to the case where $f_n$ has Gaussian coefficients. This leads us to compare the distributions of the number of sign changes of two long Gaussian vectors whose covariance matrices, denoted by A and B, are close entry-wise. We propose an elegant argument that reduces to showing $\mathcal A^{1/2} - \mathcal B^{1/2}$ is small in Frobenius norm which, in turn, is bounded by the nuclear norm of $\mathcal A-\mathcal B$ by Powers–St{\o}rmer inequality.

	\item The edge intervals $J = [0, 1 - m^{-1})$. The aforementioned method proves ineffective for edge intervals where the polynomial $f_n$ deviates significantly from stationarity. Moreover, considering each $x\in [0, 1]$, the terms $\xi_i x^{i}$ that make a substantial contribution to the sum are those with $i\in \{0, 1, \dots, \frac{O(1)}{1-x}\wedge n\}$. For instance, when $x\le 1 - 1/C$, only a finite number of terms are influential.
	
	To attain an error term of order $m^{-c}$ for the dyadic interval centered around $x$, it would necessitate this interval containing a growing function of $m$ real roots. We do not see how to achieve it considering there are $\Theta(\log m)$ dyadic intervals, which is the same for the total number of roots in $J$. We run a soft argument to show that each of the dyadic intervals can only contain $O(\log m)$ real roots with probability $1 - m^{-c}$ and so the whole edge interval contains at most $O(\log m)^{2}$ real roots. 
	
\end{itemize}

	\subsection{Outline} In Section \ref{sec:prep}, we provide preliminaries. In Section \ref{sec:bulk}, we state and prove a more general result for the bulk intervals.  In Section \ref{sec:bulk2}, we prove Theorems \ref{thm:l_a} and \ref{thm:lower}. Finally, in Sections \ref{sec:edge}, we prove Theorem \ref{thm:upper}.

		\vspace{5mm}
	{\bf Notations.} 
	We use standard asymptotic notations under the assumption that $n$ tends to infinity.  For two positive  sequences $(a_n)$ and $(b_n)$, we say that $a_n \gg b_n$ or $b_n \ll a_n$ if there exists a constant $C$ such that $b_n\le C a_n$.  
	If $a_n\ll b_n\ll a_n$, we say that $b_n=\Theta(a_n)$.  If $\lim_{n\to \infty} \frac{a_n}{b_n} = 0$, we say that $a_n = o(b_n)$.  We also write that $a_n=O_C(b_n)$ to emphasize that the implied constant depends on a given parameter $C$.
	
	\section{Preliminaries}\label{sec:prep}
	\subsection{Kac-Rice formula} This is a well-known formula to calculate the real root density of random polynomials. It relies on the fact that for any smooth functions $f$ without double roots, for any $a, b$ at which $f$ does not vanish, it holds that
	$$N_{f} ([a, b])= \lim_{\ep\to 0}\frac{1}{2\ep}\int_{a}^{b} |f'(x)|\textbf{1}_{|f(x)|\le \ep} dx.$$
	In the case that $f$ is a Gaussian process then $(f(x), f'(x))$ is a Gaussian vector and so taking expectation both sides gives ( ref. \cite{Kac1943average}, \cite[Theorem 2.5]{farahmand1996})
	\begin{equation}\label{eq:KR}
\int_{a}^{b} \rho_1(x)dx :=	\E N_{f}([a, b])  = \frac1\pi \int_{a}^{b}  \frac{\sqrt{P(x)Q(x) -R(x)^2}}{P(x)} dt,
	\end{equation}
	where $\rho_1$ is the real-zero density and  
	\ba{
		P(x):= \E (f(x)^{2}), Q(x):= \E (f'(x)^{2}), \text{ and }R(x) := \E(f(x)f'(x)).
	}
Applying this formula to $f = f_n$ with Gaussian coefficients, we get 
\begin{eqnarray}\label{eq:KR:Kac}
\rho_1(x) &=&\frac1\pi\sqrt{\frac{1}{(1-x^2)^2}-\frac{(n+1)^{2} x^{2n}}{(1-x^{2n+2})^2}}.
\end{eqnarray}
We refer also to \cite{EK} for a nice geometric interpretation of this formula.
	\subsection{Local universality} This is a useful property that allows us to reduce questions on general distributions of the $\xi_i$ to the case when they are standard Gaussian (or any other distributions with mean 0, variance 1, and bounded $(2+\ep_0)$ moments). We refer to \cite{nguyenvuCLT, nguyenvurandomfunction17} and \cite{TVpoly} and the references therein for more detailed discussions.  Here, we summarize the local universality results for the Kac polynomials that we will use later. 
	
	Firstly, the number of real roots on the real line is known to be universal, up to an error term of order $O(1)$:
	\begin{equation} 
	|\E N_{f_n}(\R) - \E_{G} N_{f_n}(\R)|=O(1)\nonumber
	\end{equation}
	where the subscript $G$ denotes the distribution when the coefficients  $(\xi_i)_{i\geq 0}$ are i.i.d. standard Gaussian random variables.
	
	This was established in \cite[Corrolary 2.6]{DOV} where its authors decomposed the real line into dyadic intervals, showed that for each of these interval, the difference in the mean number of real roots is small and then used triangle inequality. One can directly read from their proof that this holds for any interval $S\subset \R$, which means
	\begin{equation}\label{uni:real}
	|\E N_{f_n}(S) - \E_{G} N_{f_n}(S)|=O(1).
	\end{equation}
	
	Beyond the expectation, we also need a universality property on the distribution of the number of real roots in the bulk intervals. The following was proved in \cite[Theorem 2.1]{nguyenvuCLT}.
	\begin{theorem}\label{thm:uni:p}
		Assume that the $\xi_i$ have uniformly bounded $(2+\ep_0)$-moments for some $\ep_0>0$. There exist positive constants $c$ and $C$ such that for every number $0\le b_n<a_n<1$ satisfying $a_n\le (\log n)^{A}$ for all constant $A$ and for all sufficiently large $n$ (depending on $A$), for every function $F:\R \to \R$ whose derivatives up to order 3 are bounded by $1$, we have
		\ben{ 
			\Big |\E F\left( N_{f_n}(1-a_n, 1-b_n)\right) - \E_{G} F\left( N_{f_n}(1-a_n, 1-b_n)\right) \Big | \leq C a_n^{c}+Cn^{-c},\notag
		}
	\end{theorem}
	From here, for any interval $[a, b]\subset \R$, let $F$ be a smooth function with values in $[0, 1]$ and $\textbf{1}_{[\lceil a\rceil, \lfloor  b\rfloor]}\le F\le \textbf{1}_{(\lceil a\rceil -1, \lfloor  b\rfloor +1)}$. We note that for $N\in \N$, $F( N) = \textbf{1}_{N\in [a, b]}$. Thus,
	\ben{ 
		\Big |\P\left( N_{f_n}(1-a_n, 1-b_n) \in [a,b]\right) - \P_{G} \left( N_{f_n}(1-a_n, 1-b_n) \in [a,b]\right) \Big | \leq Ca_n^{c}+Cn^{-c}.\label{uni:p}
	}

	\section{The bulk}\label{sec:bulk} In this section, we study the bulk intervals. We consider sub-intervals of $[0, 1]$ of the form
	\be{
		I_{m,\ell} = [1-1/m, 1 -\ell/n]
	}
	where $m$ and $\ell$ are large parameters growing with $n$. For the interval to be non-empty, it is necessary that $m\le n/\ell$. By the Kac-Rice formula, the density of the real roots in the interval $[1-1/C, 1-1/n]$, when the coefficients are standard Gaussian, is given by
	$$\rho_1(x)= \frac{\Theta(1)}{1-x},$$
	and in the complement $[0, 1]\setminus I_{m, \ell}$, the mean number of real roots is $O( \log m+\log \ell)$. If this number is at most $a\log n$ where $a$ is a small constant, we call the interval $I_{m, \ell}$ a {\it bulk interval}. 
	
	The typical example of bulk intervals to keep in mind is $m = n^{\alpha}$ and $\ell = C\log n$ for some constants $\alpha$ and $C$.
	In general, we assume the following conditions.
	\begin{itemize}
		\item $m, \ell$ are not too big \begin{equation}\label{cond:ml1}
		m \ell \leq n^{1-\beta}, \quad \text{ for some constant } \beta>0.
		\end{equation}
		This guarantees that the interval $I_{m ,\ell}$ contains $\Theta(\log n)$ number of real roots;
		\item $m$ is not too small
		\begin{equation}\label{cond:ml2}
		m\gg (\log n)^{A}, \quad \text{ for some constant $A$};
		\end{equation}
		\item $\ell$ is not too small
		\begin{equation}\label{cond:ml3}
		\ell\ge 3\log n.
		\end{equation}
	\end{itemize}
The goal of this section is to show the following.
	\begin{theorem} \label{thm:l_ml} Assume that the $\xi_i$ have uniformly bounded $(2+\ep_0)$-moments for some $\ep_0>0$. Let $m \geq 1$ and $\ell \geq 1$ be real numbers satisfying Conditions \eqref{cond:ml1}, \eqref{cond:ml2}, and \eqref{cond:ml3}. 		 
		Then for any $\ep>0$, there exist positive constants $c$ and $C$ such that for all $n$,
		\begin{equation}\label{eq:Iml}
		\P \left( \Big| N_{n}(I_{m,\ell}) - \E N_{n}(I_{m,\ell}) \Big| \geq \ep \log n \right) \leq Cm^{-c}.
		\end{equation}
		Moreover,
		\begin{equation}\label{eq:Im}
		\P \left( \Big| N_{n}([1-1/m, 1]) - \E N_{n}([1-1/m, 1]) \Big| \geq \ep \log n \right) \leq Cm^{-c}.
		\end{equation}
	\end{theorem}
	The same result can be established for the other intervals $[-1, 0]$, $[1, \infty)$ and $(-\infty, -1]$ where the corresponding bulk intervals are $-I_{m, \ell}$, $I_{m, \ell}^{-1}$ and $-I_{m, \ell}^{-1}$. To pass the results on $[0, 1]$ to the other three intervals, we note that the number of real roots of 
	$$f_n(x) = \sum_{i=0}^{n} \xi_i x^{i}$$
	in $[-1, 0]$ is the same as the number of real roots of $\hat f_n(x):= f_n(-x) = \sum_{i=0}^{n} (-1)^{i}\xi_i x^{i}$ which is just another Kac polynomial satisfying our conditions. Similarly, the number of real roots of $f_n$ in $[1, \infty)$ is the same as the number of real roots of $\bar f_n(x):= x^{n}f_n(1/x)=\sum_{i=0}^{n} \xi_{n-i} x^{i}$ which is also a Kac polynomial. Likewise for $(-\infty, -1]$.  
	\subsection{Limiting Gaussian process} 
In preparation for the proof of Theorem \ref{thm:l_ml}, we will prove in this section that the scaling limit of the Kac polynomial $f_n(t)$ is a stationary Gaussian process $(Z_t)$. This will allow us to pass the concentration problem for the number of real roots of $f_n$ to that of $Z$.\\	
For any $x,y \in [0,1]$, we define the normalized covariance of $f_n(x), f_n(y)$ to be
\be{
	c_n(x,y)= \frac{\Cov(f_n(x),f_n(y))}{\sqrt{\Var(f_n(x)) \Var(f_n(y))}}.
}
Then by simple computations, we have 
\be{
	c_n(x,y) = \frac{G(x^n,y^n)}{G(x,y)},
}
where 
\be{
	G(x,y) = \frac{1-xy}{\sqrt{(1-x^2)(1-y^2)}}.
}	
The following lemma shows that the limit of $c_n$ is the $\sech$ function:
$$\sech(a)= \frac{1}{\cosh(a)}=\frac{2}{e^a+e^{-a}}\le 1.$$
\begin{lemma} \label{lem:cnxy}
	Let $m\ge 2, \ell\ge 1$ be real numbers satisfying $m \leq n/ \ell$. There exists a positive constant $C$ such that for all $x, y\in I_{m, \ell}$, it holds that
	\be{
		\left |c_n(x,y)-\sech\left (\frac{t-s}{2}\right ) \right | \leq C[m^{-1}+n^{2}e^{-2\ell}], 
	}
	where $s=\log (1-x)$ and $t=\log (1-y)$.	
\end{lemma}
We note that by imposing Condition \eqref{cond:ml3}, the factor $n^{2}$ in the second error term will be swallowed by $e^{-2\ell}$.
\begin{proof}
	In this proof, we will utilize several estimates in \cite[Lemmas 2.1 \& 4.1]{Dembopersistence}.  W.l.o.g. we assume that $x\leq y$.
	Set $z=1-x$ and $w=1-y$, then $0<  w \leq z\le 1/m\le 1/2$ since $m\ge 2$.
	Observe that
	\begin{equation}\label{eq:g}
	G(x,y) = \frac{z+w}{2\sqrt{zw}}\left (1 - \frac{1-\frac{zw}{z+w} - \sqrt{1-\frac{z}{2}}\sqrt{1-\frac{w}{2}}}{1-\frac{zw}{z+w}}\right )^{-1}.
	\end{equation}
	And so,
	\begin{eqnarray}
	c_n(x,y) &\geq& \frac{1}{G(x,y)} \nonumber\\
	&\geq& \frac{2\sqrt{zw}}{z+w}\left (1 - 2\left (1-\frac{zw}{z+w} - \sqrt{1-\frac{z}{2}}\sqrt{1-\frac{w}{2}}\right )\right )  \nonumber\\
	&\geq& \frac{2\sqrt{zw}}{z+w}  - \frac{4\sqrt{zw}}{z+w} \left (1- \sqrt{1-\frac{z}{2}}\sqrt{1-\frac{w}{2}}\right )  \notag\\
	&\geq& \frac{2\sqrt{zw}}{z+w} -  2\left (1-  \left (1-\frac{1}{2m}\right )\right ) = \frac{2\sqrt{zw}}{z+w} -\frac1m.\notag
	\end{eqnarray}
	Now writing $z=e^{-s}$ and $w=e^{-t}$, we get 
	\ben{ \label{locn}
		c_n(x,y) \geq \sech((t-s)/2) - 1/m.
	}
	Also by \eqref{eq:g}, we have
	\begin{equation}\label{eq:g2}
	G(x,y) \geq \frac{z+w}{2\sqrt{zw}} = \frac{1}{\sech\left (\frac{t-s}{2}\right )}.
	\end{equation}
	Since $x \leq y \leq 1-\ell/n$, we have
	$$\sqrt{1-x^{2n}}\sqrt{1-y^{2n}}\ge 1-y^{2n}\ge 1 - (1-\ell/n)^{2n}\ge 1-e^{-2},$$ 
	and
	\begin{eqnarray}
	0\le G(x^{n}, y^{n})-1=\frac{1-x^{n}y^{n}}{\sqrt{(1-x^{2n})(1-y^{2n})}}-1\ll 1-x^{n} y^{n} - \sqrt{(1-x^{2n})(1-y^{2n})}.\notag
	\end{eqnarray}
Let $\eta = 1 - \sqrt{\frac{w}{z}}$. 	For $\mathfrak s\in [0, \eta]\subset [0, 1]$, we set $p(\mathfrak s) = 1-(1-x)(1-\mathfrak s)^{2}$. Then, $p(0) = x$, $p(\eta) = y$ and $p(\cdot)$ is increasing. Let
	$$h(\mathfrak s) = 1-x^{n} p(\mathfrak s)^{n} - \sqrt{(1-x^{2n})(1-p(\mathfrak s)^{2n})}.$$
	We note that $h(0) = h'(0)=0$ and 
	\begin{eqnarray}
	h''(\mathfrak s) &=& n p''(\mathfrak s)\left (\mathcal T p(\mathfrak s)^{2n-1}-x^{n} p(\mathfrak s)^{n-1}\right )\notag\\
	&+& np'(\mathfrak s)^{2}\left ((2n-1)p(\mathfrak s)^{2n-2} \mathcal T+ \frac{np(\mathfrak s)^{4n-2}}{1-p(\mathfrak s)^{2n}} \mathcal T - (n-1)x^{n}p(\mathfrak s)^{n-2}\right )
	\end{eqnarray}
	where $\mathcal T = \sqrt{ \frac{1-x^{2n}}{1-p(\mathfrak s)^{2n}}}\le \frac{1}{\sqrt{1-y^{2n}}}\ll 1$. Therefore, since the first two derivatives of $p$ are bounded for $\mathfrak s\in [0, \eta]$, we get that
	$$\sup_{\mathfrak s\in [0, \eta]} h''(\mathfrak s)\ll \sup_{\mathfrak s\in [0, \eta]} n^2 p(\mathfrak s)^2 \leq  n^{2}y^{2n}\ll n^{2}e^{-2\ell}.$$
	This gives $h(\eta)\ll n^{2}e^{-2\ell} \eta^{2}$.
	Hence,
	$$G(x^{n}, y^{n})-1 \ll h(\eta) \ll n^{2}e^{-2\ell} \eta^{2} \leq n^{2}e^{-2\ell}.$$
	Applying this and \eqref{eq:g2} yields
	\begin{eqnarray}
	\label{uocn}
	c_n(x,y) = \frac{G(x^n,y^n)}{G(x,y)} \leq  \frac{1+Cn^{2}e^{-2\ell}}{G(x,y)} \le \sech\left (\frac{t-s}{2}\right ) + Cn^{2}e^{-2\ell},
	\end{eqnarray}
where $C$ is a positive constant.
	Combining \eqref{locn} and \eqref{uocn}, we obtain the desired result.  
\end{proof}

We will later see that on $I_{m, \ell}$, $f_n$ is close to a stationary Gaussian process with covariance being the limit of the $c_n$. We first introduce the process and several useful properties.  
\begin{definition} [Limiting Gaussian Process $(Z_t)$] 
	Let $B_u$ be the standard Brownian motion.	Consider the Gaussian process $(Z_t)_{t\geq 0}$ defined by 
	\ben{
		Z_t = \frac{\int_0^\infty h_t(u) dB_u }{\left( \int_0^\infty h^2_t(u) du \right)^{1/2}}, \text{ where\ } h_t(u) =e^{-e^{-t}u}.  \label{def:Z}
	}
\end{definition}
\begin{lemma} \label{lem:Z}
	The following assertions hold. 	
	\begin{itemize}
		\item [(i)] For any $t,s \geq 0$,
		\[ \cov(Z_t,Z_s) =\sech\left (\frac{t-s}{2}\right). \]
		\item  [(ii)] For any  $0 \leq a <b < \infty$,
		\[ \E[  N_Z([a,b]) ]:= \E [ \# \{ t \in [a,b]: Z_t =0 \} ] = \frac{b-a}{2 \pi}. \]
		\item [(iii)] We have 
		\[ \sup_{t\geq 0} \E[|Z_t''|^2]  \leq \frac{3}{2}. \]    
		\item [(iv)] The spectral density  of $(Z_t)_{t\geq 0}$ is 
		\[ p(\lambda) = \frac{1}{2 \pi} \int_0^\infty e^{-i \lambda t} r(t) dt =  \, \sech( \pi \lambda), \]         
		where $r(\cdot)$ is the covariance function given by $r(t):=\cov(Z_0, Z_t) = \sech(t/2)$.
	\end{itemize}
\end{lemma}

\begin{proof}
	We first recall that for any $a>0$ and $k\geq 0$,
	\ben{ \label{gaint}
		\int_0^\infty e^{-a u} u^k du = k! a^{-k-1}.
	}
	We set 
	$$\hat Z_t = \int_0^\infty h_t(u) dB_u.$$
	Then
	\ben{ \label{m_t}
		M_t:= \E \left (\hat Z_t^{2}\right ) =\int_0^\infty h_t^2(u) du = \int_0^\infty e^{-2e^{-t} u} du =e^t/2,
	}	
	and 
	\bea{
		\cov(Z_t,Z_s) = \E [Z_t Z_s] &=& 2e^{-(t+s)/2}  \int_0^\infty h_t(u) h_s(u) du \\
		&=& \frac{2e^{-(t+s)/2}}{e^{-t}+e^{-s}} =\sech\left (\frac{t-s}{2}\right ),
	}
	which proves (i).

	By the Kac-Rice formula \eqref{eq:KR},
	\ben{ \label{kac-rice}
		\E[N_Z([a,b])] = \E[N_{\hat Z}([a,b])]= \frac{1}{\pi} \int_a^b \frac{\sqrt{A_t M_t -B_t^2}}{M_t} dt,
	}
	where $M_t$ is as \eqref{m_t} and 
	\ba{
		A_t&:= \E (\hat Z_t'^{2})=\int_0^\infty (\partial_t h_t(u))^2 du = \int_0^\infty e^{-2e^{-t}u} e^{-2t} u^2 du, \\
		B_t &:= \E(\hat Z_t\hat Z_t') =\frac12\E\left ((\hat Z_t^2)'\right )=\frac12 M'_t.
	}
	Using \eqref{gaint}, we have
	\ben{ \label{abt} 
		A_t = B_t = e^{t}/4,
	} 
	which together with \eqref{m_t} and \eqref{kac-rice} imply that 
	\be{
		\E[N_Z([a,b])] = \frac{1}{\pi} \int_a^b \frac{e^t/4}{e^t/2} dt = \frac{b-a}{2 \pi}.
	} 
	We now prove (iii). Thanks to \eqref{m_t},  $Z_t=\sqrt{2}e^{-t/2} \int_0^\infty h_t(u) dB_u$. Therefore,
	\be{
		Z''_t = \sqrt{2}\left[ e^{-t/2} \int_0^\infty \partial_{tt} h_t(u) dB_u -\frac{e^{-t/2}}{2}  \int_0^\infty \partial_t h_t(u) dB_u  + \frac{e^{-t/2}}{4} \int_0^\infty h_t(u) dB_u \right].
	}
	Hence, using the Cauchy-Schwarz inequality $(a+b+c)^2 \leq 3(a^2+b^2+c^2)$, we get
	\bea{
		\E[|Z''_t|^2] &\leq& 6 e^{-t} \left[  \int_0^\infty (\partial_{tt} h_t(u))^2 du +  \frac14\int_0^\infty (\partial_{t} h_t(u))^2 du +  \frac{1}{16}\int_0^\infty  h_t^2(u) du \right]\\
		&=& 6 e^{-t} [\frac{e^{t}}{4}+\frac{1}{4} A_t +\frac{1}{16} M_t] \leq 6,
	}
	where we have used \eqref{m_t} and \eqref{abt}, and 
	\ba{
	 \int_0^\infty (\partial_{tt} h_t(u))^2 du &= e^{-4t}\int_0^\infty e^{-2e^{-t}u} u^4 du -2  e^{-3t}\int_0^\infty e^{-2e^{-t}u} u^3 du +  e^{-2t}\int_0^\infty e^{-2e^{-t}u} u^2 du \\
	 &= e^t/4,
    }
by applying \eqref{gaint}. 	 The above estimate holds for all $t$, so we get (iii).  
	
	The last part (iv) follows from  classical results in Fourier transforms, see e.g. \cite[Chapter 9]{titchmarsh1986introduction}. 
\end{proof}

	\subsection{Proof of Equation \eqref{eq:Iml} in Theorem \ref{thm:l_ml}}\label{proof:thm:lml}
	Applying \eqref{uni:p} where we used the Condition \ref{cond:ml2} so that the interval $I_{m, \ell}$ satisfies the hypothesis of Theorem \ref{thm:uni:p}, there exists a universal constant $c=c(\ep)>0$, such that for any $0 \leq a < b < \infty$, 
\ben{ \label{uv-ie:ml:1}
	\Big |\P \left( N_{f_n}(I_{m,\ell}) \in [a,b]\right) - \P_{G} \left( N_{f_n}(I_{m,\ell}) \in [a,b] \right) \Big | \ll m^{-c},
}	
where the subscript $G$ denotes the distribution when the coefficients  $(\xi_i)_{i\geq 0}$ are i.i.d. standard Gaussian. Therefore, we only need to consider the Gaussian case in the rest of this proof.

We assume now that  $(\xi_i)_{i\geq 0}$ are i.i.d. and have the normal law $\cN(0,1)$. Consider the change of variables $s=\log (1-x)$ (or equivalently $x=1-e^{-s}$) and define
\[ g_n(s) = \frac{f_n(1-e^{-s})}{(\Var (f_n(1-e^{-s})))^{1/2}}, \qquad s\in J_{m,\ell}:=[\log m, \log n-\log \ell].   \]
Then 
\be{ 
	N_{f_n}(I_{m,\ell}) = N_{g_n}(J_{m,\ell}),
}	
and our goal is to prove 
\ben{ \label{dpgn:2}
	\P_{G} \left( \Big| N_{g_n}(J_{m,\ell}) - \E N_{g_n}(J_{m,\ell}) \Big| \geq \ep \log n \right) \ll m^{-c}. 
}
First, by Kac-Rice formula \eqref{eq:KR:Kac},
\begin{eqnarray}
\E_{G}[N_{g_n}(J_{m,\ell})]&=&\E_{G}[N_{f_n}(I_{m,\ell})] = \int_{I_{m, \ell}} \rho_1(x)dx\notag
\end{eqnarray}	
with
\begin{eqnarray}
	\rho_1(x) &=&\frac1\pi\sqrt{\frac{1}{(1-x^2)^2}-\frac{(n+1)^{2} x^{2n}}{(1-x^{2n+2})^2}} = \frac{1+o(1)}{2\pi} \frac{1}{1-x}  
\end{eqnarray}
where we used Conditions \eqref{cond:ml2} and \eqref{cond:ml3}.
This gives
\begin{eqnarray}
\label{enfn:2}
\E_{G}[N_{g_n}(J_{m,\ell})]&=&\frac{1+o(1)}{2\pi}\int_{I_{m, \ell}} \frac{1}{1-x}dx= \frac{(1+o(1))\log (n/m\ell)}{2 \pi}.
\end{eqnarray}	
We now decompose the interval $J_{m, \ell}$ into equally small intervals and approximate the number of real roots in each interval by the number of size changes. For a small parameter $\delta_m$ to be chosen, we partition the interval $J_{m, \ell}$ by points
\[  s_k =  \log m + k \delta_m, \quad  k=0, \ldots\]
The number of such points is 
\[T=T_{m, \ell} =  \delta_m^{-1} (\log n -\log m - \log \ell)=\Theta(\delta_m^{-1}\log n) \text{ by Condition \eqref{cond:ml1}}.\]
Let $Z$ denote the Gaussian process in \eqref{def:Z}. We define the number of sign changes of $g_n$ and $Z$ on $J_{m, \ell}$ as follows.
\bea{
	N^\star_{g_n} = \sum_{k=0}^{T-1} \textbf{1}_{g_n(s_k)g_n(s_{k+1})<0}, \quad N^\star_Z = \sum_{k=0}^{T-1} \textbf{1}_{Z_{s_k} Z_{s_{k+1}}<0}.
}   
To handle the concentration of $N_{g_n}$, we try to reduce that to the concentration of $N_{Z}$ which has been established by Basu--Dembo--Feldheim--Zeitouni \cite{Basuconcentration}. To do the comparison, we show that $N_{g_n} \approx N^\star_{g_n} \approx N^\star_Z \approx N_{Z}$. In particular, we show the following 3 lemmas. The first lemma is a concentration result of $N_{Z}$.
\begin{lemma}  \label{lm:donz}There exists $c=c(\ep)>0$, such that 
	\begin{eqnarray}
	\P \left( \Big| N_Z(J_{m,\ell}) - \frac{ \log (n/m \ell)}{2 \pi} \Big| \geq \frac{\ep \log n}{2} \right) \le n^{-c}.
	\end{eqnarray}
\end{lemma} 
Next, we compare the number of real roots and the number of sign changes. 
\begin{lemma} \label{lm:nens} We have
	\begin{eqnarray}
	\P_{G}[ N_{g_n}(J_{m,\ell}) \neq N^\star_{g_n}] + \P[ N_Z(J_{m,\ell}) \neq N^\star_{Z}]  \ll  \delta_m^{1/3} \log n.
	\end{eqnarray}
\end{lemma}
Finally, we compare the distribution of the number of sign changes of $g_{n}$ with that of $Z$.
\begin{lemma} 	\label{lm:ngzs}	For all  $0 \leq a < b < \infty$, we have
	\begin{eqnarray}
	\left  |\P _{G}\left( N_{g_n}^\star  \in (a,b) \right) - \P \left( N^\star_Z  \in (a,b) \right) \right  | \ll  \delta_m^{1/3}\log n+ (\log n)^{5/2}\delta_m^{-31/6} (1/m +n^{2}e^{-2\ell}).
	\end{eqnarray}

\end{lemma}
By setting $\delta_m$ to be $m^{-1/6}$ (for instance), all of these probabilities are bounded by 
$$m^{-2c}(\log n)^{5/2}+(\log n)^{5/2}m^{31/36}  n^{2}e^{-2\ell}\ll m^{-c}+ e^{-\ell}\ll m^{-c} +n ^{-c}\ll m^{-c},$$ 
for sufficiently large $n$, where we used Conditions \eqref{cond:ml1}, \eqref{cond:ml2} and \eqref{cond:ml3},	proving Theorem \ref{thm:l_ml}.

\begin{proof} [Proof of Lemma \ref{lm:donz}]
	
	By Lemma \ref{lem:Z} (iv), $p(\lambda)=\sech(\pi \lambda)$, the spectral density of  the Gaussian process $(Z_t)_{t\geq 0}$ decays exponentially, and thus the conditions \cite[Eq. (1.4) \& (1.5)]{Basuconcentration} hold. More precisely, for $\kappa =2$ and $\kappa_0 =1/2$,
	\be{
		\int_{\R} e^{|\lambda| \kappa} p(\lambda) d \lambda < \infty,
	}
	and 
	\be{
		\int_{\R} |r(t,\kappa_0)| dt < \infty, \quad r(t,\kappa_0)= \int_{R} \cos(t\lambda) \cosh(2 \kappa_0 \lambda) p(\lambda) d\lambda.
	}
	Therefore, by 
	\cite[Theorem 1.1]{Basuconcentration}
	\be{
		\P \left( \Big| N_Z(J_{m,\ell}) - \E[N_Z(J_{m,\ell}) ] \Big| \geq (\ep \log n)/4 \right) \leq  n^{-c}.   
	}
	Notice further that $\E[N_Z(J_{m,\ell})] = ((2 \pi)^{-1}+o(1)) \log (n/m \ell)$ by Lemma \ref{lem:Z} (ii).  Hence, Lemma \ref{lm:donz} follows.
\end{proof}

\begin{proof} [Proof of Lemma \ref{lm:nens}] We will prove later that the probability of having more than 1 root in a small interval is small. In particular, we will show that there exists a constant $C$ such that for all $0<\delta<\frac{1}{2C}$ and  $x, y\in I_{m, \ell}$ satisfying  $\log \left(\tfrac{1-x}{1-y}\right) =\delta$, and for all $i\ge 2$,
	\ben{ \label{multipleroot}
		\P_{G}(N_{f_n}([x,y])\ge 2)\le (C\delta)^{3/2}.		}
	If we write $x_k=1-\exp(-s_k)$, then $N_{g_n}([s_k,s_{k+1}])=N_{f_n}([x_{k+1},x_k])$ and $\log \left(\tfrac{1-x_{k+1}}{1-x_k}\right) =\delta_m$.  Hence, thanks to \eqref{multipleroot}, for all $k=0, \ldots, T-1$, 
	\ben{ \label{ngsk}
		\P_{G}(N_{g_n}([s_k,s_{k+1}])\ge 2)\ll \delta_m^{3/2}.
	}
	We notice that  if $N_{g_n} (J_{\ep, n}) \neq N^\star_{g_n}$, there must be some $k$ such that $N_{g_n} ([s_k,s_{k+1}])\neq \textbf{1}_{g_n(s_k)g_n(s_{k+1})<0}$. Hence, by the union bound,
	$$\P_{G}(N_{g_n} (J_{\ep, n})\neq N_{g_n}^\star) \le \sum_{k=0}^{T-1} \P_{G}\left [N_{g_n} ([s_k,s_{k+1}])\neq \textbf{1}_{g_n(s_k)g_n(s_{k+1})<0}\right ].$$
	Furthermore,  if $N_{g_n}([s_k,s_{k+1}]) \in \{ 0, 1\}$, then  $N_{g_n} ([s_k,s_{k+1}])= \textbf{1}_{g_n(s_k)g_n(s_{k+1})<0}$. Thus, by the union bound and \eqref{ngsk} and taking $s=3/4$,
	\begin{equation} \label{ngns}
	\P_{G}(N_{g_n} (J_{m, n})\neq N_{g_n}^\star) \le \sum_{k=0}^{T-1}  \P_{G}(N_{g_n} ([s_k,s_{k+1}])\ge 2)\le \sum_{k=0}^{T-1}  \delta_m^{3/2}\ll T \delta_m^{3/2}.
	\end{equation}
	By the same arguments as above, we have 
	\begin{eqnarray*}
		\P(N_{Z} (J_{m, \ell})\neq N_{Z}^\star) &\le&  \sum_{k=0}^{T-1}  \P(N_{Z} ([s_k,s_{k+1}])\ge 2) = T\pp \left( N_Z[0, \delta_m] \geq 2\right),
	\end{eqnarray*}
	since $(Z_t)_{t\geq 0}$ is stationary. For any $a >0$, we will show that 
	\ben{\label{multipleroot2}
		\pp \left( N_Z[0, \delta_m] \geq 2\right) \le \frac{\delta_m ^4}{a^2} \max_{t \in [0, \delta_m]} \E [|Z_t''|^2] + \pp\left( |Z_{\delta_m}| \leq a \right).
	} 
	Assuming this, by Lemma \ref{lem:Z} (iii), 
	\[ \max_{t \in [0, \delta_m]} \E [|Z_t''|^2] \leq 3/2. \]
	Moreover, since $Z_{\delta_m} \sim \cN(0,1)$,
	\be{
		\pp (|Z_{\delta_m}| \leq a) \ll a.
	}  
	Thus by letting $a=\delta_m ^{4/3}$, we have
	\be{
		\pp \left( N_Z[0, \delta_m] \geq 2\right) \ll \delta_m^4 a^{-2} +a  \ll \delta_m ^{4/3}.
	} 
	Thus,
	\begin{equation}\label{key}
	\P[ N_Z(J_{m,\ell}) \neq N^\star_{Z}]  \ll T\delta_m ^{4/3}.\notag
	\end{equation}
	Combining this with \eqref{ngns}, we get 
	\begin{eqnarray}
	\P_{G}[ N_{g_n}(J_{m,\ell}) \neq N^\star_{g_n}] + \P[ N_Z(J_{m,\ell}) \neq N^\star_{Z}]  \ll  T\delta_m^{4/3} \ll \delta_m^{1/3} \log n
	\end{eqnarray}
	proving Lemma \ref{lm:nens}. 
\end{proof}

\begin{proof}[Proof of \eqref{multipleroot}] By Rolle's theorem and the fundamental theorem of calculus, if $f_n$ has at least $2$ zeros in the interval $(x, y)$ then
	$$
	|f_n(y)|\leq\int_{x}^{y}\int_{x}^{t}|f_n^{''}(s)|dsdt=:K_{x, y}.
	$$
	Therefore,
	\begin{align*}
	\P_{G}(N_{f_n}([x,y])\ge 2)\quad\leq\quad& \P \left(K_{x, y}\geq\varepsilon_{1}\sqrt{V(y)}\right)+ \P \left(|f_n(y)|\leq\varepsilon_{1}\sqrt{V(y)}\right)
	\end{align*}
	where $\varepsilon_{1}$ to be chosen and 
	\begin{equation}\label{varbound:1}
	V(y)   =\Var f_n(y)=\sum_{i=0}^{n} y^{2i}= \frac{\Theta(1)}{1-y+1/n}\quad \forall y\in(1-1/C, 1).
	\end{equation}	
	By Gaussianity, we have
	$$
	\P \left(|f_n(y)|\leq\varepsilon_{1}\sqrt{V(y)}\right)\ll \ep_1.
	$$
Letting $\varepsilon_{1}:=(C\delta)^{3/2}$,	it remains to show that
	\begin{equation}\label{intbound}
	\P \left(K_{x, y}\geq\varepsilon_{1}\sqrt{V(y)}\right)\ll (C_1\delta)^{3/2}.
	\end{equation}	
By Markov's inequality, we have
	\begin{eqnarray}
	&&\left(\varepsilon_{1}\sqrt{V(y)}\right)^{3}\P \left(K_{x, y}\geq\varepsilon_{1}\sqrt{V(y)}\right) \le  \E \left(\int_{x}^{y}\int_{x}^{t}|f_n^{''}(s)|dsdt\right)^{3}.\nonumber
	\end{eqnarray}
	By H\"older's inequality, the right-hand side is at most
	$$(y-x)^{4}\E \int_{x}^{y}\int_{x}^{t}|f_n^{''}(s)|^{3}dsdt\le  (y-x)^{6}\max_{w\in[x,y]}\E |f_n^{''}(w)|^{3}.$$
	For each $w\in (x, y)$, since $f_n^{(k)}(w)$ is a Gaussian random variable, using the hypercontractivity inequality for the Gaussian distribution (see, for example, \cite[Corollary 5.21]{massartconcentrationbook}), we have 
	\begin{eqnarray}
	\E |f_n^{''}(w)|^{3}&\ll& \left(\E |f_n^{''}(w)|^{2}\right)^{\frac{3}{2}}
\le \left (\sum_{j=0}^{n} j^{4} y^{2j}\right)^{\frac{3}{2}}
	\ll  \left (\frac{1}{(1-y+1/n)^{5}}\right )^{\frac{3}{2}}.\nonumber
	\end{eqnarray}
Thus,
	\begin{equation*}
	\P \left(K_{x, y}\geq\varepsilon_{1}\sqrt{V(y)}\right)\ll  \frac{ (1-y+1/n)^{\frac{3}{2}}}{\varepsilon_{1}^{3}} \frac{(y-x)^{6}}{(1-y+1/n)^{\frac{15}{2}}} \le\frac{ 1}{\varepsilon_{1}^{3}}\left(\frac{y-x}{1-y}\right)^{6}.
	\end{equation*}
	Since $\frac{y-x}{1-y} = \frac{1-x}{1-y}-1 = e^{\delta}-1\le 2\delta$, we get
	\begin{equation} 
	\P \left(K_{x, y}\geq\varepsilon_{1}\sqrt{V(y)}\right) \ll\frac{\delta^{6}}{\ep_1^{3}} \ll \delta^{3/2}. \nonumber
	\end{equation}
	This completes the proof of \eqref{multipleroot}.
\end{proof}

\begin{proof}
	[Proof of \eqref{multipleroot2}] We use a similar argument as in the proof of \eqref{multipleroot} with $f_n$ being replaced by $Z$ and $[x, y]$ by $[0, \delta_m]$. In particular, we get
		\begin{eqnarray}
			\pp \left( N_Z[0, \delta_m] \geq 2\right) &\le& \P\left (|Z(\delta_m)|\le a\right )+\P\left (K_{0, \delta_m}\ge a \right )\notag
			\end{eqnarray}
			where now, $K_{0, \delta_m} = \int_{0}^{\delta_m}\int_{0}^{t}|Z^{''}(s)|dsdt$. By Markov's inequality,
			\begin{eqnarray} 
			\P\left (K_{0, \delta_m}\ge a \right ) &\le& \frac{1}{a^2}\E \left(\int_{0}^{\delta_m}\int_{0}^{t}|Z^{''}(s)|dsdt\right)^{2}\le  \frac{\delta_m ^4}{a^2} \max_{t \in [0, \delta_m]} \E [|Z_t''|^2]\nonumber
		\end{eqnarray}
		giving \eqref{multipleroot2}.
\end{proof}
\begin{proof}[Proof of   Lemma \ref{lm:ngzs}] 		 	
	Define 
	\be{
		d_n(t,s)=\cov(g_n(s),g_n(t))=\E_{G}[g_n(s)g_n(t)].
	}
	Then by Lemma \ref{lem:cnxy}, 
	\ben{ \label{cnts}
		|d_n(s,t)- \sech((t-s)/2)| \leq C[1/m +n^{2}e^{-2\ell}],
	}
	for some universal constant $C$. We define 
	$$  U=(Z_{s_k})_{k=0}^T, \quad  V=(g_n(s_k))_{k=0}^T. $$
	Then $U$ and $V$ are two centered Gaussian random vectors. Let us denote the covariance matrices of $U$ and $V$ by $\mathcal A=(a_{i,j})_{0 \leq i,j \leq T}$ and $\mathcal B=(b_{i,j})_{0 \leq i,j \leq T}$ respectively. Then by \eqref{cnts}, 
	\be{
		|a_{i,j}-b_{i,j}| \leq C[1/m +n^{2}e^{-2\ell}].
	}
	Thus 
	\ben{ \label{na-b}
		||\mathcal A-\mathcal B||_F = \left(\sum_{i,j=0}^T |a_{i,j}-b_{i,j}|^2 \right)^{1/2} \ll T[1/m +n^{2}e^{-2\ell}].
	}
	We observe that  since $\mathcal A$ and $\mathcal B$ are symmetric and positive semi-definite,
	\ben{
		U \overset{(d)}{=}\mathcal A^{1/2} W, \quad V \overset{(d)}{=}\mathcal B^{1/2} W, \text{ where } W \sim \cN(0,Id).
	}
	Therefore, the desired result Lemma \ref{lm:ngzs} is a consequence of the following: 
	\ben{ \label{engz}
		\pp \left( \sum_{i=0}^{T-1} \I_{\sign(X_i X_{i+1}) \neq \sign(Y_i Y_{i+1})} \geq 1 \right) \leq  \delta_m^{1/3}\log n+ (\log n)^{5/2}\delta_m^{-31/6} (1/m +n^{2}e^{-2\ell})
	}   
	where $X= \mathcal A^{1/2} W$ and $Y= \mathcal B^{1/2} W$. We have
	\bean{ \label{s12}
		&&\pp \left( \sum_{i=0}^{T-1} \I_{\sign(X_i X_{i+1}) \neq \sign(Y_i Y_{i+1})} \geq 1 \right) \notag \\
		& \leq& \pp \left( \sum_{i=0}^{T} \I_{\sign(X_i) \neq \sign(Y_i)} \geq 1 \right) \notag  \\
		&\leq& \pp \left( \sum_{i=0}^{T} \I_{|X_i| \leq  \delta_m^{4/3}} \geq 1 \right) +  \pp \left( \sum_{i=0}^{T} \I_{|X_i-Y_i| \geq  \delta_m^{4/3}} \geq 1\right) \notag 	 \\
		&\leq&    \sum_{i=0}^{T} \pp(|X_i| \leq  \delta_m^{4/3})  +  \pp \left( \sum_{i=0}^{T} (X_i-Y_i)^2 \geq  \delta_m^{8/3} \right) =S_1 +S_2.
	}
	Since $\pp(|X_i| \leq \delta_m^{4/3}) \ll \delta_m^{4/3}$, 
	\ben{ \label{bs1}
		S_1 \ll T \delta_m^{4/3}.
	}
	For the second term, observe that
	\be{
		\sum_{i=0}^{T} (X_i-Y_i)^2 =\sum_{i=0}^{T} ((\mathcal A^{1/2} W)_i-(\mathcal B^{1/2} W)_i)^2 = \norm{\Gamma W}_2^2,
	}
	where 
	\be{
		\Gamma = \mathcal A^{1/2}-\mathcal B^{1/2}.
	}
	Furthermore,
	\be{
		\norm{\Gamma W}_2^2  \overset{(d)}{=} \norm{\cN(0,\Gamma^t\Gamma)}_2^2 \overset{(d)}{=} \norm{\cN(0,D)}_2^2  \overset{(d)}{=} \sum_{i=0}^T \lambda_i \chi^2_i,
	}
	where $(\lambda_i)_{i=0}^T$ are eigenvalues of $\Gamma^t\Gamma$ and $D=\textrm{diag}(\lambda_0,\ldots, \lambda_T)$, and $(\chi^2_i)_{i\geq 0}$ are i.i.d. chi-squared  random variables. Therefore, 
	\begin{eqnarray}
	S_2&=&\P\left( \sum_{i=0}^{T} (X_i-Y_i)^2 \geq  \delta_m^{8/3} \right)  =  \P\left( \sum_{i=0}^{T} \lambda_i \chi^2_i \geq  \delta_m^{8/3} \right)\\
	&\leq& \P\left( \text{there exists } 0\le i\le T: \chi_i^2 \geq \frac{\delta_m^{8/3}}{\sum_{i=0}^{T} \lambda_i}  \right) \notag\\
	&\le& (T+1)    \P\left(  \chi^2 \geq \frac{\delta_m^{8/3}}{\sum_{i=0}^{T} \lambda_i} \right)\notag   \\
	&\ll& T  \delta_m^{-8/3} \sum_{i=0}^{T} \lambda_i =   T  \delta_m^{-8/3} \text{Tr}(\Gamma^{t}\Gamma) = T  \delta_m^{-8/3} \norm{\Gamma}_F^2,
	\end{eqnarray}\label{s2}
	where in the last expression was the Frobenius norm.
	By Powers–St{\o}rmer inequality  \cite[Lemma 4.1]{powers1970free}
	\be{
		\norm{\Gamma}_F^2 =	\norm{\mathcal A^{1/2}-\mathcal B^{1/2}}^2_F \leq \norm{\mathcal A-\mathcal B}_{\textrm{nuc}},
	}
	where on the right hand side is the nuclear norm (or trace norm) of the matrix, which is sum of its singular values.
	Notice also that 
	\begin{eqnarray*}
		\norm{\mathcal A-\mathcal B}_{\textrm{nuc}} =\sum_{i=0}^{T} \sigma_i(\mathcal A-\mathcal B) \le \sqrt{T+1} \sqrt{\sum_{i} (\sigma_i(\mathcal A-\mathcal B))^{2}}
		&=& \sqrt{T+1}\norm{\mathcal A-\mathcal B}_{F}.
	\end{eqnarray*}
	Hence, combining above estimate and \eqref{na-b}  we get
	\be{
		\norm{\Gamma}_F^2 \leq \sqrt{T+1} \norm{\mathcal A-\mathcal B}_{F} \ll T^{3/2}(1/m +n^{2}e^{-2\ell}).
	}
	This together with \eqref{s2} implies
	\ben{ \label{bs2}
		S_2 \ll T^{5/2} \delta_m^{-8/3} (1/m +n^{2}e^{-2\ell}).
	} 
	Combining \eqref{s12}, \eqref{bs1} and \eqref{bs2}, we obtain
	\ben{ 
		\pp \left( \sum_{i=0}^{T-1} \I_{\sign(X_i X_{i+1}) \neq \sign(Y_i Y_{i+1})} \geq 1 \right) \ll   T\delta_m^{4/3}+T^{5/2} \delta_m^{-8/3} (1/m +n^{2}e^{-2\ell}) \notag
	}   
	proving \eqref{engz} because $T = \Theta(\delta_m^{-1}\log n)$.
\end{proof} 

	\section{Proof of Theorem \ref{thm:lower}, Equation \eqref{eq:Im} of Theorem \ref{thm:l_ml} and Theorem \ref{thm:l_a}}\label{sec:bulk2}
	Since the proofs are related, making use of Theorem \ref{thm:l_ml}, we put them together in this section.
	\subsection{Proof of Theorem \ref{thm:lower}}
	As before, it suffices to study the roots in $[0, 1]$.	We apply  Theorem \ref{thm:l_ml} to the following interval 
\be{
	I=I_{m, \ell},  \text{ where $ m=  n^{\ep/4}$ and $\ell = 3\log n$}.
}
In the Gaussian case, we have by \eqref{enfn:2},
\begin{eqnarray}
\E_{G}[N_{f_n}(I_{m,\ell})] 
&=& \frac{(1+o(1))\log (n/m\ell)}{2 \pi} \notag\\
&=&\frac{(1+o(1))(1-\ep/3)\log n}{2 \pi}\ge \E_{G} N_{f_n}([0,1]) - \frac{\ep}{3}\log n\notag
\end{eqnarray}	
for sufficiently large $n$.\newline
We recall \eqref{uni:real} which states
$$ |\E N_{f_n}(S) - \E_{G} N_{f_n}(S)|=O(1).$$
Therefore,
\begin{eqnarray}
\E [N_{f_n}(I_{m,\ell})]  &\ge& \E_G [N_{f_n}(I_{m,\ell})]  - O(1)\notag\\
&\ge& \E_{G} N_{f_n}([0,1]) - \frac{\ep}{3}\log n - O(1)\ge   \E N_{f_n}([0,1]) - \frac{\ep}{2}\log n.\notag
\end{eqnarray}	
and thus   Theorem \ref{thm:l_ml} gives that
\begin{eqnarray}
\P \left(  N_{f_n}([0,1]) \leq  \E N_{f_n}([0,1])  -\ep \log n \right) &\le& \P \left(  N_{f_n}(I_{m, \ell}) \leq    \E [N_{f_n}(I_{m,\ell})] -\frac{\ep}{2} \log n \right) \notag\\
& \ll& n^{-c},\notag
\end{eqnarray}
for some positive constant $c$.

	\subsection{Proof of Equation \eqref{eq:Im} in Theorem \ref{thm:l_ml}} The interval $[1-1/m, 1]$ is a union of an  interval of the form $I_{m, \ell}$ and the remaining interval close to 1. We set
	$$\ell = 3\log n, \text{and }  S_{\ell}=[1-\ell/n,1].$$
	 Then
	 $$[1-1/m, 1] = I_{m, \ell}\cup S_{\ell}.$$
	 By the same argument as in the proof of Theorem \ref{thm:lower}, we get
	 \begin{eqnarray}
	 \P \bigg( N_{f_n}([1-1/m, 1]) \le \E N_{f_n}([1-1/m, 1]) - \ep \log n \bigg) \ll n^{-c}.\notag
	 \end{eqnarray}
	 Thus, it remains to consider the upper tail. 
	We have 
	\begin{eqnarray}
	&&\P \bigg( N_{f_n}([1-1/m, 1]) \ge \E N_{f_n}([1-1/m, 1])+\ep \log n \bigg) \notag\\
	&\le& 	\P \bigg( N_{f_n}(I_{m, \ell}) \ge \E N_{f_n}(I_{m, \ell}) + \ep\log n/2  \bigg) + \P\bigg (N_{f_n}(S_{\ell})\ge \ep\log n/2\bigg ).\notag
	\end{eqnarray}
The first term on the right is at most $m^{-c}$ by Theorem \ref{thm:l_ml}. Thus, we need to show that 
	\begin{equation}\label{eq:Sell}
	\P\left (N_{f_n}(S_{\ell})\ge \ep\log n/2\right )\ll n^{-c}.
	\end{equation}
To prove \eqref{eq:Sell}, we apply \eqref{uni:p} to obtain a constant $c=c(\ep)>0$, such that for any $0 \leq a < b < \infty$, 
\ben{ 
	\Big |\P \left( N_{f_n}(S_{\ell}) \in (a,b) \right) - \P_{G} \left( N_{f_n}(S_{\ell}) \in (a,b) \right) \Big | \ll \ell^{c}n^{-c}\ll m^{-c}.\nonumber
}	

Thus, we can reduce to the Gaussian setting. Using the same strategy as in proving \eqref{multipleroot}, we will prove the following lemmas which bound the probability of having too many real roots in a small interval. The first lemma deals with dyadic interval of the form $[1-2\delta, 1 - \delta]$.
\begin{lemma} \label{lem:ngm}
	Assume that the random variables $\xi_i$ are iid standard Gaussian. 
	Suppose that $x \leq y \leq 1$ and $h \in \N$ such that 
	$y-x=1-y$ and $h(1-y) \leq 1$. Then
	\begin{equation}
	\P_{G}  (N_{f_n}([x ,y])\geq h)\ll  \left (\frac{nh(1-y)}{4^h}\right )^{1/3}.\notag
	\end{equation}
\end{lemma}
The next lemma deals with intervals containing 1.
\begin{lemma} \label{lem:ngk}
	Assume that the random variables $\xi_i$ are iid standard Gaussian. 
	Suppose that $0 \leq y \leq 1$ and $h \in \N$ such that 
	$\sqrt{n} \geq h \geq 3(1-y)n $. Then
	\begin{equation} 
	\P _{G} (N_{f_n}([y ,1])\geq h)\ll 2^{-h/3}.\notag
	\end{equation}
\end{lemma}
Assuming these lemmas, we decompose $S_\ell$ into $i_*+1$ intervals using  the following points 
\be{
	x_i=1- \frac{\ell}{2^i n}, \quad 0 \leq i \leq i_*:= \lceil \log_2(36/\eps) \rceil.
}
We have $x_{i_*}\ge 1-\frac{\ep\log n}{12 n}$.
Applying Lemma \ref{lem:ngm} to $x=x_i$, $y=x_{i+1}$ for $i<i_*$ and $h=\ep_1\log n$, where $\ep_1=\frac{\ep}{4i_*}$, gives
\ben{
	\P_{G}  (N_{f_n}([x_i ,x_{i+1}])\geq \ep_1\log n) \ll \left (\frac{nh(1-x_{0})}{4^h}\right )^{1/3}  \ll n^{-c}. 
}
Applying Lemma \ref{lem:ngk} to $y=x_{i_*}$ and $h=\eps \log n/4$, we have 
\ben{
	\P_{G}  (N_{f_n}([x_{i_*}, 1])\geq \ep\log n/4) \ll 2^{-h/3}\ll n^{-c}. 
}
By the union bound,
\begin{eqnarray}
\P_{G}  (N_{f_n}([x_{0}, 1])\geq \ep\log n/2)&\le& \sum_{i=0}^{i_*-1} \P_{G}  (N_{f_n}([x_i ,x_{i+1}])\geq \ep_1\log n)\notag\\
&&+ \P_{G}  (N_{f_n}([x_{i_*}, 1])\geq \ep\log n/4)\ll n^{-c}. \notag
\end{eqnarray}
\begin{proof}[Proof of Lemma \ref{lem:ngm}]
	By Rolle's theorem and the fundamental theorem of calculus, if $f_n$ has at least $h$ zeros in the interval $[x,y]$ then
	$$
	|f_n(y)|\leq\int_{x}^{y}\int_{x}^{y_{1}}\cdots\int_{x}^{y_{h-1}}|f_n^{(h)}(y_{h})|dy_{h}\ldots dy_{1}=:\mathcal J_{n}.
	$$
	Since $f_n(y)$ is Gaussian with mean 0 and Variance $\frac{1-y^{2n+2}}{1-y^2}=O((1-y)^{-1})$ for $y\in S_\ell$, we have for all $\theta>0$,
	$$
	\P \left(|f_n(y)|\leq\frac{\theta}{\sqrt{1-y}}\right)\ll \theta. 
	$$
	Therefore,
	\ben{ \label{ngm}
		\P_{G}  (N_{f_n}([x ,y])\geq h) \leq \P \left(|f_n(y)|\leq \frac{\theta}{\sqrt{1-y}}\right)+ \P \left(\mathcal J_n\geq \frac{\theta}{\sqrt{1-y}}\right)\le \theta  + \frac{1-y}{\theta^2} \E(\mathcal J_n^2).
	}
	By Cauchy-Schwarz inequality, 
	\bea{
		\E[\mathcal J_n^2] &\leq&  \int_{x}^{y}\int_{x}^{y_{1}}\cdots\int_{x}^{y_{h-1}}  dy_{h}\ldots dy_{1}  \times \int_{x}^{y}\int_{x}^{y_{1}}\cdots\int_{x}^{y_{h-1}}\E[f_n^{(h)}(y_{h})^{2}]dy_{h}\ldots dy_{1} \\
		&=&  \frac{(y-x)^h}{h!} \int_{x}^{y}\int_{x}^{y_{1}}\cdots\int_{x}^{y_{h-1}}\E[f_n^{(h)}(y_{h})^{2}]dy_{h}\ldots dy_{1}.
	}
	Next, we consider
	\begin{eqnarray*}
		&&\int_{x}^{y}\int_{x}^{y_{1}}\cdots\int_{x}^{y_{h-1}}\E \left (f_n^{(h)}(y_{h})^{2}\right )dy_{h}\ldots dy_{1}\\
		&=& \sum_{i=h}^{n} \left(\frac{i!}{(i-h)!} \right)^2  \int_{x}^{y}\int_{x}^{y_{1}}\cdots\int_{x}^{y_{h-1}} y_h^{2i-2h}dy_{h}\ldots dy_{1}\\
		&\le& \sum_{i=h}^{n}\left(\frac{i!}{(i-h)!} \right)^2 \int_{x}^{y}\int_{x}^{y_{1}}\cdots\int_{x}^{y_{h-2}} \frac{y_{h-1}^{2i-2h+1}}{2i-2h+1}dy_{h-1}\ldots dy_{1}\\		 
		&\le& \sum_{i=h}^{n} \left(\frac{i!}{(i-h)!} \right)^2 \frac{(2i-2h)!}{(2i-h)!} y^{2i-h}  \\
		& \ll & \sum_{i=h}^{n} \frac{i}{i-h+1} \sqrt{\frac{2i-2h+1}{2i-h}} \frac{(i/e)^{2i}}{((i-h)/e)^{2i-2h}} \frac{((2i-2h)/e)^{2i-2h}}{((2i-h)/e)^{2i-h}}y^{2i-h}\\
		&\ll& h (4ye)^{-h} \sum_{i=h}^{n} \frac{(2i)^{2i}}{(2i-h)^{2i-h}} y^{2i}, 
	\end{eqnarray*}
	where for the fifth line, we have used Stirling formula. 		Let us consider $g(z)=2z \log (2z) - (2z-h) \log(2z-h) + 2z \log y $ with $h \leq z \leq n$. Then $g'(z)=2 \log \left( \frac{2z y}{2z-h} \right)$. Hence, $g'(z)=0$ iff $z=z_*:=\frac{h}{2(1-y)}$ and we also have 
	\be{
		\sum_{i=h}^{n} \frac{(2i)^{2i}}{(2i-h)^{2i-h}} y^{2i} \leq n \frac{(2z_*)^{2z_*}}{(2z_*-h)^{2z_*-h}} y^{2z_*}=n \left( \frac{hy}{1-y} \right)^h.
	} 
	Therefore,
	\bea{
		\E[\mathcal J_n^2] &\ll& n h (4ye)^{-h} \frac{(y-x)^h}{h!}	\left( \frac{hy}{1-y} \right)^h\\
		&\ll& n h 4^{-h},
	}	
	by using Stirling formula and $y-x=1-y$.
	Combining this with \eqref{ngm}, we obtain
	\be{
		\P_{G}  (N_{f_n}([x ,y])\geq h) \ll \theta  + \frac{1-y}{\theta^2}n h 4^{-h} \ll   \left (\frac{nh(1-y)}{4^h}\right )^{1/3}.
	}
\end{proof}

\begin{proof} [Proof of Lemma \ref{lem:cnxy}] We use the same arguments for the previous lemma. 
	By Rolle's theorem, if $f_n$ has at least $h$ zeros in the interval $[y,1]$ then
	$$
	|f_n(1)|\leq\int_{y}^{1}\int_{y}^{y_{1}}\cdots\int_{y}^{y_{h-1}}|f_n^{(h)}(y_{h})|dy_{h}\ldots dy_{1}=:\mathcal I_{n}.
	$$
	Since $f_n(1)$ is Gaussian with mean 0 and Variance $n+1$, we have
	$$
	\P \left(|f_n(1)|\leq\theta \sqrt{n}\right)\le \theta.
	$$
	Therefore,
	\ben{ \label{ngk}
		\P_{G}  (N_{f_n}([y ,1])\geq h) \leq \P \left(|f_n(1)|\leq \theta\sqrt{n}\right)+ \P \left(\mathcal I_n\geq \theta\sqrt{n}\right)\le \theta + \frac{1}{n\theta^2} \E(\mathcal I_n^2).
	}
	By Cauchy-Schwarz inequality, 
	\bea{
		\E[\mathcal I_n^2] &\leq&    \frac{(1-y)^h}{h!} \int_{y}^{1}\int_{y}^{y_{1}}\cdots\int_{y}^{y_{h-1}}\E[f_n^{(h)}(y_{h})^{2}]dy_{h}\ldots dy_{1}.
	}
	We have
	\begin{eqnarray*}
		&&\int_{y}^{1}\int_{y}^{y_{1}}\cdots\int_{y}^{y_{h-1}}\E \left (f_n^{(h)}(y_{h})^{2}\right )dy_{h}\ldots dy_{1}
		\le \sum_{i=h}^{n} \left(\frac{i!}{(i-h)!} \right)^2 \frac{(2i-2h)!}{(2i-h)!} \\
		& \le & n \left(\frac{n!}{(n-h)!} \right)^2 \frac{(2n-2h)!}{(2n-h)!}\ll  n (n/2)^h,
	\end{eqnarray*}
	where we have used Stirling formula. 		Therefore,
	\bea{
		\E[\mathcal I_n^2] &\ll&  \frac{n}{2^h} \left( \frac{en(1-y)}{h} \right)^h\leq \frac{n}{2^h}
	}	
	using Stirling formula for $h!$ and the assumption $h \geq 3n(1-y)$.  Combining this with \eqref{ngk} gives
	\be{
		\P _{G} (N_{f_n}([y ,1])\geq h) \ll \theta + \frac{1}{\theta^2 2^h} \ll  2^{-h/3}
	}
	concluding the proof.	\end{proof}
			\subsection{Proof of Theorem \ref{thm:l_a}} This is just a special case of Equation \eqref{eq:Im} of Theorem \ref{thm:l_ml} where $m = n^{a}$ (and $\ell = 3\log n$, say).
		
	\section{Proof of Theorem \ref{thm:upper}} \label{sec:edge}
	If $N_n\ge \E N_n +\ep \log n$ then there exists at least one of the four intervals $[0, 1]$, $[1, \infty)$, $[-1, 0]$, and $(-\infty, -1]$ for which the number of real roots $N$ is at least $\E N + \ep\log n/4$. Therefore, by symmetry and the union bound, it suffices to show that
	\be{
		\pp\left( N_{n}([0, 1])  \geq \E N_n([0, 1]) +\ep  \log n \right) \leq C\exp \left(-c \sqrt{\log n}\right).
	}

	We let $m=e^{\sqrt{\ep_1 \log n}}$ where $\ep_1$ is a constant depending only on $\ep$ and $T_{m}=[0,1-1/m]$. By \eqref{eq:Im} of Theorem \ref{thm:l_ml}, we have
	$$\pp\left( N_{f_n}([1-m^{-1}, 1]) \geq \E N_{f_n}([1-m^{-1}, 1])+\frac{\ep}{2}  \log n \right)\ll m^{-c}.$$
	Thus, it remains to show that
	\begin{equation}\label{eq:T}
	\pp\left (N_{f_n}(T_m) \geq \frac{\ep}{2}  \log n\right )\ll m^{-c}.
	\end{equation}
	We will choose $\ep_1$ so that $\frac{\ep}{2}  \log n\ge \frac{6}{(\log 2)(\log (5/4))}(\log m)^{2}$.

	To prove \eqref{eq:T}, we will, just like $S_{\ell}$, decompose the interval $T_m$ into dyadic intervals and show that each interval has a small probability of having too many roots. However, the strategy in proving Lemma \ref{lem:ngm} and \ref{lem:ngk} does not give a strong enough bound for $T_m$ while only applies for Gaussian random variables, so we will resort to a different argument. We decompose the interval $T_m$ into dyadic intervals by endpoints $x_i = 1-2^{-i}$ where $i=0, \dots, K+1$ where $K = \lfloor\log_{2}m\rfloor$.
		It suffices to show that 
	\begin{equation}\label{eq:Ij}
	\pp\left (N_{f_n}([x_j, x_{j+1}]) \geq  \frac{6}{\log (5/4)}\log m\right )\ll m^{-c}.
	\end{equation}
	Indeed, by the union bound over $j$, we shall get
	\begin{eqnarray}
	\pp\left (N_{f_n}(T_m) \geq   \frac{6}{(\log 2)(\log(5/4))}(\log m)^{2}\right ) &\le& \P\left (\exists j:  N_{f_n}([x_j, x_{j+1}]) \geq   \frac{6}{\log (5/4)}\log m\right )\notag\\
	&\ll& K m^{-c}\ll m^{-c/2}\notag
	\end{eqnarray}	
	proving \eqref{eq:T}.
To prove \eqref{eq:Ij},	for each $j\in \{0, \dots, K\}$, we apply the classical Jensen's inequality (see e.g. \cite[Appendix A.1]{nguyenvurandomfunction17} for a rather elementary proof) which states that for every analytic function $f$, for every ball $B(z, r)$ in the complex plane, and for every $R>r$, we have
	\ben{ 
		N_{f} (B(z, r)) \le \frac{\log \max_{w\in B(z, R)}|f(w)|-\log \max_{w\in B(z, r)}|f(w)| }{\log \frac{R^2+r^2}{2Rr}}.\notag
	}
	Applying this inequality to $f_n$, we obtain
	\ben{ \label{n1ma}
		N_{f_n} ([x_j, x_{j+1}]) \le  \frac{1}{\log (5/4)}\log (M_j/m_j),
	}
	where  $R_j = x_{j+1}-x_j$, $r_j = R_j/2$,  and 
	$$m_j =|f_n\left (x_{j+1}\right )|, \quad M_j = \max\left \{|f_n(x)|: \left |x-\frac{x_{j+1}+x_j}{2}\right |\le R_j\right \} .$$
To prove \eqref{eq:Ij}, by \eqref{n1ma}, it suffices to show that
	\begin{equation}\label{eq:Mj}
		\pp(M_j \geq  m^{3})\ll m^{-c}
	\end{equation}
	and
	\begin{equation}\label{eq:mj}
		\pp(m_j \leq  m^{-3})\ll m^{-c}.
	\end{equation}
	
	\begin{proof}[Proof of \eqref{eq:Mj}]	
		Let $\mathcal A$ be the event that for all $i\le n$, 
		$$|\xi_i|\le m \left (1+\frac{1}{4m}\right )^{i}.$$
		By Chebyshev's inequality, this event happens with probability at least 
		$$1 - O(1)\sum_{i=0}^{\infty} m^{-2} \left (1+\frac{1}{4m}\right )^{-2i}\geq 1- O(m^{-1}).$$
Remark that 	 $x\le  1 - \frac{1}{4m}$ for all $x$ such that $\left |x-\frac{x_{j+1}+x_j}{2}\right |\le R_j$.  Hence, under the event $\mathcal A$
		\begin{eqnarray*}
			M_j \le  \sum_{i=0}^{n} |\xi_i| \left (\frac{x_{j+1}+x_j}{2}+R_j\right )^{i}\le  m \sum_{i=0}^{\infty}  \left (1+\frac{1}{4m}\right )^{i} \left (1-\frac{1}{4m}\right )^{i} \le m^{3}
		\end{eqnarray*}
		which proves \eqref{eq:Mj}.
	\end{proof}
	
	Finally, to prove \eqref{eq:mj}, we shall use the following result (Lemma~9.2 in \cite{TVpoly}) which can be traced back to Littlewood and Offord. 
	\begin{lemma}\label{lm:anti} Assume that $\xi_i$ are independent random variables with unit variance and uniformly bounded $(2+\ep_0)$-moments for some $\ep_0>0$. Then there are constants $C',\alpha$ such that for any complex numbers $a_0,\ldots ,a_n$ containing a lacunary subsequence $|a_{\ell_1}|\ge 2|a_{\ell_2}|\ge \ldots \ge 2^k|a_{\ell_k}|$, we have 
		$$
		\P\left(\left  |\sum_{i=0}^n a_i\xi_i\right |\le |a_{\ell_k}|  \right)\le C'e^{-\alpha k}.
		$$ 
	\end{lemma}
	\begin{proof}[Proof of \eqref{eq:mj}]
		We apply Lemma \ref{lm:anti} to $f_n(x_{j+1}) = \sum_{i=0}^{n} x_{j+1}^{i} \xi_i$. So, $a_i:=x_{j+1}^{i}$ is a decreasing sequence in $i$. We will find by a greedy algorithm a lacunary sequence with $a_{\ell_k}\approx m^{-3}$. Let $\ell_1=0$ and for $i\ge 1$, define
		$$
		\ell_{i+1}=\min\{t>\ell_{i}\; : \; a_{\ell_i}\ge 2 a_{t}\}.
		$$		
		Let $k$ be such that $a_{\ell_{k+1}}<m^{-3}\le a_{\ell_k}$. Since $a_0=1\ge m^{-3}$ and 
		$$a_n = x_{j+1}^{n}\le (1-1/(4m))^{n}\le m^{-3}/2,$$
		it holds that $\ell_{k+1}\le n$.
		
		We observe that for all $i$,
		$$\frac{a_i}{a_{i+1}} = x_{j+1}^{-1}\le x_1^{-1} = 2.$$
		So, 
		$$2\le \frac{a_{\ell_i}}{a_{\ell_{i+1}}} \le \frac{a_{\ell_i}}{a_{\ell_{i+1}-1}} \times 2\le 4$$
		which gives
		$$2^{k+1}\le \frac{a_{0}}{a_{\ell_{k+1}}}\le 4^{k+1}.$$
		Since $a_0=1$, we have $4^{-k-1}\le a_{\ell_{k+1}} < m^{-3}$. Thus, by Lemma \ref{lm:anti},
		\begin{eqnarray}
			\P(m_j\le m^{-3})&\le& \P\left (\left |\sum_{i=0}^{n} a_i \xi_i\right |\le m^{-3}\right )\le \P\left (\left |\sum_{i=0}^{n} a_i \xi_i\right |\le a_{\ell_k}\right )\nonumber\\
			&\ll& e^{-\alpha k} \ll 4^{-ck}\le m^{-c'}\nonumber
		\end{eqnarray}
		for some constants $c$ and $c'$. This completes the proof.
	\end{proof}

	\section{Acknowledgments} We thank Hoi Nguyen for valuable comments on the presentation of the manuscript. We thank Gr\'egory Schehr and Oleg Zaboronski for helpful references.

	\bibliographystyle{plain}
	\bibliography{polyref}

\end{document}